\documentclass[12pt,reqno]{amsart}
\usepackage{amsfonts,amsmath,amssymb,amsthm,color,enumerate,graphics,hyperref,latexsym,mathrsfs,mathtools,stmaryrd,xspace}

\newcommand{\A}{A}

\newcommand{\AS}{\textup{AS}\xspace}
\newcommand{\Aut}{\mathrm{Aut}}

\newcommand{\calC}{\mathcal{C}}
\newcommand{\calO}{\mathcal{O}}
\newcommand{\Cay}{\mathrm{Cay}}
\newcommand{\CD}{\textup{CD}\xspace}
\newcommand{\Cen}{\mathbf{C}}
\newcommand{\Co}{\mathrm{Co}}
\newcommand{\Cos}{\mathrm{Cos}}

\newcommand{\GammaSp}{\mathrm{\Gamma Sp}}
\newcommand{\GL}{\mathrm{GL}}

\newcommand{\M}{\mathrm{M}}
\newcommand{\magma}{\textsc{Magma}\xspace}

\newcommand{\Nor}{\mathbf{N}}
\newcommand{\Out}{\mathrm{Out}}
\newcommand{\PA}{\textup{PA}\xspace}
\newcommand{\PGL}{\mathrm{PGL}}
\newcommand{\PGammaL}{\mathrm{P\Gamma L}}

\newcommand{\POmega}{\mathrm{P\Omega}}
\newcommand{\ppd}{\mathrm{ppd}}

\newcommand{\PSL}{\mathrm{PSL}}

\newcommand{\PSp}{\mathrm{PSp}}
\newcommand{\PSU}{\mathrm{PSU}}
\newcommand{\Rad}{\mathrm{Rad}}

\newcommand{\SD}{\textup{SD}\xspace}

\newcommand{\SL}{\mathrm{SL}}
\newcommand{\Soc}{\mathrm{Soc}}

\newcommand{\Sp}{\mathrm{Sp}}
\newcommand{\SU}{\mathrm{SU}}

\newcommand{\Sy}{S}
\newcommand{\Z}{Z}
\newcommand{\ZZ}{\mathbf{Z}}

\newtheorem{theorem}{Theorem}[section]

\newtheorem{lemma}[theorem]{Lemma}
\newtheorem{proposition}[theorem]{Proposition}
\theoremstyle{definition}

\newtheorem{example}[theorem]{Example}

\newtheorem{hypothesis}[theorem]{Hypothesis}

\newtheorem{question}[theorem]{Question}

\begin{document}

\title{Vertex-primitive $s$-arc-transitive Cayley digraphs}

\thanks{Corresponding author: Binzhou Xia}
\thanks{This work was supported by NNSFC (12431013), NNSFC (12571362), NSF of Guangxi (2025GXNSFAA069013) and Melbourne Research Scholarship.}

\author[J.J. Li]{Jing Jian Li}
\author[Y.T. Shi]{Yong Tang Shi}
\author[Y. Wang]{Yu Wang}
\author[B. Xia]{Binzhou Xia}

\address{J.J. Li\\School of Mathematics, Center for Applied Mathematics of Guangxi, Guangxi University, Nanning, Guangxi, 530004, P.R.China}
\email{lijjhx@gxu.edu.cn}
\address{Y.T. Shi\\Center for Combinatorics and LPMC, Nankai University, Tianjin, 300071, P.R.China}
\email{shi@nankai.edu.cn}
\address{Y. Wang\\ Center for Combinatorics and LPMC, Nankai University, Tianjin, 300071, P.R.China; School of Mathematics, Center for Applied Mathematics of Guangxi, Guangxi University, Nanning, Guangxi, 530004, P.R.China}
\email{wangyu97@mail.nankai.edu.cn}
\address{B. Xia\\School of Mathematics and Statistics, University of Melbourne, Parkville, VIC 3010, Australia}
\email{binzhoux@unimelb.edu.au}

\begin{abstract}
Determining an upper bound on $s$ for vertex-primitive $s$-arc-transitive digraphs has been an open problem of considerable interest since a question asked by Praeger in 1990. Although much progress has been made and an upper bound is conjectured to be $2$, a complete classification for $s=2$ remains out of reach. In this paper, we prove that the tight upper bound on $s$ for finite vertex-primitive $s$-arc-transitive Cayley digraphs is exactly $2$. Furthermore, we completely characterize the structure of these digraphs when $s=2$.


\end{abstract}

\maketitle

\section{Introduction}
A \emph{digraph} $\Gamma$ is a pair $(V,\rightarrow)$ with a set $V$ of vertices and an antisymmetric irreflexive binary relation $\rightarrow$ on $V$.
Denote by $V(\Gamma)$ and $\Aut(\Gamma)$ the vertex set and the full automorphism group of $\Gamma$ respectively.
For a nonnegative integer $s$, an \emph{$s$-arc} of $\Gamma$ is an $(s+1)$-tuple $(v_0,v_1,\ldots,v_s)$ of vertices such that $v_{i-1}\rightarrow v_i$ for each $i\in\{1,\ldots,s\}$.
For a subgroup $G$ of $\Aut(\Gamma)$, we say that $\Gamma$ is \emph{$G$-vertex-primitive} if $G$ acts primitively on $V(\Gamma)$ (that is, $G$ does not preserve any nontrivial partition of $V(\Gamma)$), and that $\Gamma$ is \emph{$(G,s)$-arc-transitive} if $G$ acts transitively on the set of $s$-arcs of $\Gamma$.
It is easy to see that a vertex-transitive $s$-arc-transitive digraph with $s\geq 1$ is also $(s-1)$-arc-transitive.
A digraph $\Gamma$ is simply called \emph{vertex-primitive} if it is $\Aut(\Gamma)$-vertex-primitive, and \emph{$s$-arc-transitive} if it is $(\Aut(\Gamma),s)$-arc-transitive.

A well-known result of Weiss~\cite{Weiss1981} establishes the tight upper bound $s\leq7$ for finite $s$-arc-transitive graphs of valency at least $3$ (note that $s$ can be arbitrarily large for valency $2$ as cycles are $s$-arc-transitive for all $s\geq0$).
In contrast, there exist infinite families of finite $s$-arc-transitive digraphs with unbounded $s$ other than directed cycles (see, for example,~\cite[Theorem~2.8~and~Proposition~2.11]{Praeger1989}).
However, all such examples are imprimitive, and the existence of finite vertex-primitive $2$-arc-transitive digraphs aside from directed cycles remained unknown until Giudici, Li and Xia~\cite{GLX2017} constructed the first family of such digraphs in 2017.
To date, no finite vertex-primitive $3$-arc-transitive digraphs have been discovered.
Motivated by this construction, the following question is posed in~\cite{GLX2017}.

\begin{question}\label{Question}
Is there an upper bound on $s$ for finite vertex-primitive $s$-arc-transitive digraphs that are not directed cycles?
\end{question}

A group $G$ is said to be \emph{almost simple} if there exists a nonabelian simple group $T$ such that $T\trianglelefteq G\leq \Aut(T)$.
The O'Nan-Scott theorem classifies primitive permutation groups into eight types (see~\cite[Theorem~7.8]{PS2018}; see also~\cite{LPS1988} for the original proof), with almost simple groups constituting one of these classes.
A systematic investigation in~\cite[Corollary~1.6]{GX2018} of these O'Nan-Scott types has reduced Question~\ref{Question} to the case of almost simple groups.
More precisely, it was shown that an upper bound on $s$ for finite vertex-primitive $s$-arc-transitive digraphs $\Gamma$ with an almost simple automorphism group $\Aut(\Gamma)$ also serves as an upper bound for all finite vertex-primitive $s$-arc-transitive digraphs, excluding directed cycles.
Since then, this question has been resolved for several infinite families of almost simple groups: alternating and symmetric groups~\cite{CCGLPX2025,PWY2020}, linear groups~\cite{GLX2019}, symplectic groups~\cite{CGP2025}, exceptional groups of Lie type except for $\mathrm{E}_7$ and $\mathrm{E}_8$~\cite{CGP2023,YC2026}, and most sporadic groups~\cite[Lemma~4.4]{YFX2023}.

Although significant progress has been made, completely resolving Question~\ref{Question} for all almost simple groups remains an ongoing project.
Moreover, a review of these results indicates that determining a classification, or even the existence, of $2$-arc-transitive digraphs in these cases is currently out of reach.

Recall that a graph or digraph $\Gamma$ is said to be \emph{Cayley} on a group $R$ if $\Aut(\Gamma)$ contains a subgroup isomorphic to $R$ that acts regularly on $V(\Gamma)$.
It is widely believed that the vast majority of vertex-transitive graphs and digraphs in general are Cayley (see~\cite{MP1994} and~\cite[\S8.4]{MS2021}).
Thus, to further our understanding of Question~\ref{Question}, it is natural to focus on Cayley digraphs.
Another motivation, as can be seen by Proposition~\ref{PropSDCD}, is that the infinite families of vertex-primitive $2$-arc-transitive digraphs in~\cite[Theorem~1.2~and~Corollary~1.4]{GX2018} are Cayley.
In this paper, we determine the tight upper bound on $s$ for finite vertex-primitive $s$-arc-transitive Cayley digraphs and classify all such digraphs for $s\geq2$.
Recalling the definition of the direct product of digraphs from Section~\ref{Sec2}, our main result is stated as follows.

\begin{theorem}\label{ThmMain}
Let $\Gamma$ be a finite vertex-primitive Cayley digraph.
Then $\Gamma$ is $2$-arc-transitive if and only if one of the following holds:
\begin{enumerate}[{\rm(a)}]
\item\label{ThmMainA} $\Gamma$ is a directed cycle;
\item\label{ThmMainB} $\Gamma$ is isomorphic to the direct product of $m$ copies of $\Gamma(T)$ for some nonabelian simple group $T$ and positive integer $m$, where $\Gamma(T)$ is the digraph defined in Example~$\ref{ExampleSD}$.
\end{enumerate}
In particular, there exist no vertex-primitive $3$-arc-transitive Cayley digraphs other than directed cycles.
\end{theorem}

The proof of Theorem~\ref{ThmMain} proceeds by the primitive type.
Let $\Gamma$ be a $G$-vertex-primitive $(G,2)$-arc-transitive Cayley digraph that is not a directed cycle.
By~\cite[Theorem~3.1]{Praeger1989}, the primitive type of $G$ can only be Almost Simple (\AS), Simple Diagonal (\SD), Compound Diagonal (\CD), or Product Action (\PA).
In Section~\ref{SecASSDCD}, for type \AS, we prove that no such vertex-primitive $2$-arc-transitive Cayley digraph exists, a result stated as Proposition~\ref{PropAS}.
Moreover, we establish in the same section that $G$ is of type \SD or \CD if and only if $\Gamma$ belongs to a specific family of digraphs constructed from those given in Example~\ref{ExampleSD}; this is formulated as Proposition~\ref{PropSDCD}.
In Section~\ref{SecPA}, we address the case where $G$ is of type \PA, proving that no vertex-primitive $2$-arc-transitive Cayley digraph exists in this case either, which is formulated as Proposition~\ref{PropPA}.

\section{Preliminaries}\label{Sec2}

From now on, all digraphs and groups are assumed to be finite.
For a positive integer $n$ and a prime number $p$, denote by $\pi(n)$ the set of prime divisors of $n$, and if $X$ is a group then $\pi(|X|)$ is simply denoted by $\pi(X)$.
For a positive integer $n$, denote by $\Z_n$ the cyclic group of order $n$, and by $D_{2n}$ the dihedral group of order $2n$ if $n\geq 3$.
For a group $G$, denote by $\Soc(G)$ its socle (that is, the product of all minimal normal subgroups of $G$), and denote by $G^{(\infty)}$ the first perfect group in the derived series of $G$.
For two groups $K$ and $H$, denote by $K.H$ an extension of $K$ by $H$, and by $K\rtimes H$ a split extension.
For group-theoretic terminology not mentioned here, we refer the reader to~\cite{LPS1990}.

For a digraph $\Gamma$ with vertex set $U$ and a digraph $\Sigma$ with vertex set $V$, their \emph{direct product}, denoted by $\Gamma\times\Sigma$, is the digraph with vertex set $U\times V$ such that $(u_1,v_1)\rightarrow (u_2,v_2)$ if and only if $u_1\rightarrow u_2$ and $v_1\rightarrow v_2$, where $u_1,u_2\in U$ and $v_1,v_2\in V$.
For a positive integer $m$, we denote by $\Sigma^m$ the direct product of $m$ copies of $\Sigma$.

\subsection{The digraph $\Gamma(T)$}
\ \vspace{1mm}

Let $G$ be a group, let $H$ be a subgroup of $G$, let $V$ be the set of right cosets of $H$ in $G$, and let $g\in G\setminus H$ such that $g^{-1}\notin HgH$.
Define a binary relation $\to$ on $V$ by letting $Hx\to Hy$ if and only if $yx^{-1}\in HgH$ for any $x,y\in G$.
Then $(V,\to)$ is a digraph, denoted by $\Cos(G,H,g)$.
The following construction associated with nonabelian simple groups gives a family of digraphs that contribute to the examples of vertex-primitive $2$-arc-transitive digraphs.

\begin{example}[{\cite[Construction 3.1]{GX2018}}]\label{ExampleSD}
Let $T=\{t_1,\ldots,t_k\}$ be a nonabelian simple group of order $k$, let $D=\{(t,\ldots,t)\mid t\in T\}<T^k$, let $V$ be the set of right cosets of $D$ in $T^k$, and let $g=(t_1,\ldots,t_k)$.
Denote $\Gamma(T)=\Cos(T^k,D,g)$.
\end{example}

The majority of the following lemma is from~\cite[Theorem~1.2]{GX2018}.

\begin{lemma}\label{SDisCayley}
The digraph $\Gamma(T)$ in Example~$\ref{ExampleSD}$ is a vertex-primitive $2$-arc-transitive Cayley digraph and is not $3$-arc-transitive.
\end{lemma}

\begin{proof}
By~\cite[Theorem~1.2]{GX2018}, $\Gamma(T)$ is a vertex-primitive $2$-arc-transitive digraph that is not $3$-arc-transitive.
According to~\cite[Lemma~3.5]{GX2018}, $\Aut(\Gamma(T))$ contains a subgroup $X$ that is a primitive permutation group with socle $M\cong T^k$, where $M$ is the permutation group on $V$ induced by the right multiplication.
Let 
\[
R=\{(t_1,\ldots,t_{k-1},1)\mid t_1,\ldots,t_{k-1}\in T\}\leq M\trianglelefteq X.
\]
It is straightforward to verify that $M=RD$ and $R\cap D=1$.
Hence $R$ acts regularly on $V$, and so $\Gamma(T)$ is a Cayley digraph, as required.
\end{proof}

\subsection{Technical lemmas}
\ \vspace{1mm}

An expression of a group $H$ as the product of two subgroups $A$ and $B$ of $H$ is called a \emph{factorization} of $H$, where $A$ and $B$ are called \emph{factors}.
The following lemma can be easily derived from~\cite[Lemma~2.2]{GX2018}.

\begin{lemma}\label{LemPrimeFactn}
Let $\Gamma$ be a connected $(G, 2)$-arc-transitive digraph, let $u\rightarrow v\rightarrow w$ be a $2$-arc of $\Gamma$, and let $|G_v|=\prod_{i=1}^n p_i^{f_i}$, where $p_1,\ldots,p_n$ are pairwise distinct primes and $f_1,\ldots,f_n$ are positive integers.
Then the following statements hold:
\begin{enumerate}[{\rm (a)}]
\item\label{LemPrimeFactnA} $G_v=G_{uv}G_{vw}$;
\item\label{LemPrimeFactnB} $|G_{uv}|=|G_{vw}|$ is divisible by $\prod_{i=1}^n p_i^{\lceil{f_i/2}\rceil}$;
\item\label{LemPrimeFactnC} $G_{uv}$ and $G_{vw}$ are conjugate in $G$ but not conjugate in $G_v$.
\end{enumerate}
\end{lemma}

A digraph $\Gamma$ is said to be \emph{$k$-regular} if the sets of the in-neighbors and the out-neighbors of $v$ have the same size $k$ for all $v\in V(\Gamma)$, and in this case, $k$ is called the \emph{valency} of $\Gamma$.

\begin{lemma}[{\cite[Lemma~2.13]{GLX2019}}]\label{LemVal}
Every vertex-primitive $1$-arc-transitive digraph is either a directed cycle of prime length or has valency at least $3$.
\end{lemma}

For a group $X$, let $\Rad(X)$ denote the largest solvable normal subgroup of $X$.

\begin{lemma}\label{LemRglr}
Let $G$ be a permutation group on a set $V$ such that each transitive core-free subgroup $Y$ of $G$ satisfies both of the following:
\begin{enumerate}[{\rm (a)}]
\item $(Y/\Rad(Y))^{(\infty)}$ is a nonabelian simple group of order greater than $|V|$;
\item there exists $p\in\pi(G)$ such that $|(Y/\Rad(Y))^{(\infty)}|_p=|G|_p=|V|_p$.
\end{enumerate}
Then there exists no regular subgroup of $G\wr\Sy_m$ for $m\geq2$, where $G\wr\Sy_m$ acts on $V^m$ in product action.
\end{lemma}

\begin{proof}
Suppose for a contradiction that $G\wr\Sy_m$ has a regular subgroup $R$.
Let $\sigma\colon G\wr\Sy_m\to\Sy_m$ be the natural homomorphism modulo the base group $G^m$, and let $\rho_i\colon G^m\to G$ be the projection into the $i$-th direct factor.
Take an arbitrary orbit $\Delta$ of $R^\sigma$ on $\{1,\ldots,m\}$, and without loss of generality, assume $\Delta=\{1,\ldots,k\}$.
Then $R\leq G\wr(\Sy_k\times\Sy_{m-k})$.
Define a homomorphism
\[
\varphi\colon G\wr(\Sy_k\times\Sy_{m-k})\to G\wr\Sy_k,\ \ (g_1,\ldots,g_m)\alpha\beta\mapsto(g_1,\ldots,g_m)\alpha,
\]
where $(g_1,\ldots,g_m)\in G^m$, $\alpha\in\Sy_k$ and 
$\beta\in\Sy_{m-k}$.
By~\cite[Proposition~2.5]{LPS2000}, for each $i\in\{1,\ldots,k\}$, the group $R_i\coloneqq(R^\varphi\cap G^k)^{\rho_i}$ is transitive on $V(\Sigma)$.
Then each $R_i$ has a unique nonsolvable composition factor $T_i$, and $R^\varphi\cap G^k$ is a subdirect subgroup of $R_1\times\cdots\times R_k$.
Moreover, since $R^\sigma$ is transitive on $\Delta$, the groups $R_1,\ldots,R_k$ are isomorphic.

Let $N=\Rad(R_1\times\cdots\times R_k)$.
Then $R^\varphi\cap N\trianglelefteq R^\varphi\cap(R_1\times\cdots\times R_k)=R^\varphi\cap G^k$, and for each $i\in\{1,\ldots,k\}$,
\begin{align*}
\Big(\big((R^\varphi\cap G^k)\big/(R^\varphi\cap N)\big)^{(\infty)}\Big)^{\rho_i}
&=\big((R^\varphi\cap G^k)^{(\infty)}\big/(R^\varphi\cap G^k)^{(\infty)}\cap(R^\varphi\cap N)\big)^{\rho_i}\\
&=\big((R^\varphi\cap G^k)^{(\infty)}\big)^{\rho_i}\big/\big((R^\varphi\cap G^k)^{(\infty)}\cap N\big)^{\rho_i}\\
&=(R_i)^{(\infty)}\big/\big((R_i)^{(\infty)}\cap\Rad(R_i)\big)
=T_i.
\end{align*}
Thus, $\left((R^\varphi\cap G^k)\big/(R^\varphi\cap N)\right)^{(\infty)}$ is a subdirect subgroup of $T_1\times\cdots\times T_k$.
Since $R^\sigma$ is transitive on $\Delta$, it follows from Scott's Lemma (see, for example,~\cite[Theorem~4.16]{PS2018}) that the unique nonsolvable composition factor $T$ of $R^\varphi\cap G^k$ has multiplicity $\ell$ dividing $k$, where $T\cong T_1\cong\cdots\cong T_k$.
In particular, $|T|=|(R_1/\Rad(R_1))^{(\infty)}|>|V|$.

Suppose $\ell<k$.
Then $k\geq2$ and $\ell\leq k/2$.
Since $R^\varphi$ is transitive on $V^k$ and $|G|_p=|V|_p$,
\[|R^\varphi|_p
\geq|V|_p^k
=|G|_p^k.
\]
On the other hand, since $|G|_p=|T|_p$, we have $|R_1|_p=\cdots=|R_k|_p=|T|_p$, and so
\[
|R^\varphi|_p
\leq|R^\varphi\cap G^k|_p(k!)_p
=|T|_p^\ell(k!)_p
=|G|_p^\ell(k!)_p.
\]
Hence $(k!)_p\geq|G|_p^{k-\ell}\geq|G|_p^{k/2}$, which is impossible.

Therefore, $\ell=k$, and so $|R^\varphi\cap G^k|\geq|T|^k$.
Since $\Delta$ is taken to be an arbitrary orbit of $R^\sigma$ on $\{1,\ldots,m\}$, we conclude that $|R|\geq|T|^m>|V|^m$, contradicting that $R$ is regular.
This completes the proof.
\end{proof}

For integers $a\geq 2$ and $m\geq 2$, a prime number $r$ is called a \emph{primitive prime divisor} of the pair $(a,m)$ if $r$ divides $a^m-1$ but does not divide $a^i-1$ for any positive integer $i<m$.
The following result is the famous Zsigmondy's theorem.

\begin{theorem}[{\cite[Theorem IX.8.3]{BH1982}}]\label{BH1982}
Let $a\geq 2$ and $m\geq 2$ be integers.
Then $(a,m)$ has a primitive prime divisor if and only if either $a+1$ is a power of $2$ and $m=2$, or $(a,m)=(2,6)$.
\end{theorem}

For a prime power $a=p^f$ and an integer $m\geq2$, where $p$ is a prime and $f$ is a positive integer, let $\ppd(a,m)$ denote the set of primitive prime divisors of $(p,fm)$, and set $\ppd(2,6)=\{7\}$.
Note from Fermat's Little Theorem that for each $r\in\ppd(a,m)$, we have $r\equiv 1\pmod{fm}$, and so $r>fm$.

\subsection{Strategy to prove Theorem~\ref{ThmMain}}
\ \vspace{1mm}

As outlined in the Introduction, the proof of Theorem~\ref{ThmMain} proceeds as follows.
Let $\Gamma$ be a $G$-vertex-primitive $(G,2)$-arc-transitive Cayley digraph that is not a directed cycle, and let $v$ be a vertex of $\Gamma$.
By~\cite[Theorem]{LPS1988}, the primitive type of $G$ is one of \AS, \SD, \CD, or \PA, which we analyze separately.
Among others, an important tool in our analysis is the characterization of the vertex-stabilizer $G_v$, which admits a factorization as in Lemma~\ref{LemPrimeFactn}.
Furthermore, in the case where $(G_v)^{(\infty)}$ is quasisimple, this factorization is described in~\cite[Lemma~3.3]{YFX2023}.

In Section~\ref{SecASSDCD}, we investigate the cases where $G$ has type \AS, \SD, or \CD.
First assume that $G$ is of type \AS.
Since the primitive group $G$ has a regular subgroup $R$, the possible triples $(G,G_v,R)$ can be read off from~\cite[Tables~16.1--16.3]{LPS2010}.
We show that this case is impossible for the infinite families in Lemma~\ref{LemASInf} and the remaining triples in Lemma~\ref{LemASFin}, combined with computation in \magma\cite{magma} when necessary.
Therefore, there exists no $G$-vertex-primitive $(G,2)$-arc-transitive Cayley digraph with $G$ almost simple.
Next assume that $G$ is of type \SD or \CD.
By~\cite[Theorem~1.2~and~Corollary~1.4]{GX2018}, a $G$-vertex-primitive $(G,2)$-arc-transitive digraph where $G$ is of type \SD or \CD is isomorphic to the direct product of $m$ copies of $\Gamma(T)$ given in Example~$\ref{ExampleSD}$, and such a digraph is not $3$-arc-transitive.
We prove that all such digraphs are Cayley, leading to a classification of \SD or \CD type vertex-primitive $2$-arc-transitive Cayley digraphs in Proposition~\ref{PropSDCD}.

In Section~\ref{SecPA}, we assume that $G$ is of type \PA.
We first show that, by~\cite[Corollary~3(iv)]{LPS2000}, it suffices to consider $G$-vertex-primitive $(G,2)$-arc-transitive digraphs where $G$ is almost simple and contains a transitive core-free subgroup $X$.
The existence of $X$ restricts the candidates for the pair $(G,G_v)$, which can be determined using~\cite[Corollary]{LPS1996}.
Moreover, if $(G_v)^{(\infty)}$ is quasisimple, then~\cite[Lemma~3.3]{YFX2023} provides an even stronger restriction.
We observe that for most of the candidates, the group $X$ has a unique nonsolvable composition factor $T$, and $T$ satisfies the condition that $|T|>|G|/|G_v|$ and that there exists a prime $p\in\pi(G)\setminus\pi(G_v)$ with $|T|_p=|G|_p$.
A key lemma, stated as Lemma~\ref{LemRglr}, then asserts that $\Aut(\Gamma)$ has no regular subgroups, yielding a contradiction.
For the single case where $(\Soc(G),(\Soc(G)_v)^{(\infty)})=(\PSL_4(q),\PSp_4(q))$, which cannot be ruled out through this method as we cannot assure that $X$ is nonsolvable, we prove that $G_{uv}$ and $G_{vw}$ are not conjugate in $G$.
This is achieved by showing that the centralizers of $(G_{uv})^{(\infty)}$ and $(G_{vw})^{(\infty)}$ in $G^{(\infty)}$ are not isomorphic, which contradicts Lemma~\ref{LemPrimeFactn}\eqref{LemPrimeFactnC}.
Summarizing these results, we conclude that $G$ cannot be of type \PA, which is formulated as Proposition~\ref{PropPA}.

\section{Types \AS, \SD and \CD}\label{SecASSDCD}

Throughout this section, let $\Gamma=\Cay(R,S)$ be a $(G,s)$-arc-transitive Cayley digraph that is not a directed cycle, where $s\geq 2$ and $G$ is an almost simple group acting primitively on $V(\Gamma)$, let $L=\Soc(G)$, and let $v\in V(\Gamma)$.

We begin with the \AS type.
According to~\cite[Theorem~1.1]{LPS2010}, the triple $(G,G_v,R)$ lies in~\cite[Tables~16.1--16.3]{LPS2010}.
We deal with the infinite families in these tables in Lemma~\ref{LemASInf} and the remaining candidates in Lemma~\ref{LemASFin}.

\begin{lemma}\label{LemASInf}
Suppose that $(G,G_v,R)$ is as described in rows~\emph{1--2} of~\cite[Table~16.1]{LPS2010} or rows~\emph{1--4} of~\cite[Table~16.2]{LPS2010}.
Then $s=1$.
\end{lemma}

\begin{proof}
Suppose for a contradiction that $s\geq 2$, so that $\Gamma$ is $(G,2)$-arc-transitive.

\smallskip
\noindent\textbf{Case~1:}
$L=\PSL_n(q)$, $L_v= P_1$ or $P_{n-1}$, and $R$
is metacyclic of order $(q^n-1)/(q-1)$, as in row~1 of~\cite[Table~16.1]{LPS2010}.
Since $G_v$ is a maximal parabolic subgroup, this is not possible by~\cite[Lemma~2.13]{YFX2023}.

\smallskip
\noindent\textbf{Case~2:}
$L=\PSL_2(q)$ and $G_v\cap L=D_{q+1}$ for $q\equiv 3\pmod{4}$, as in row~2 of~\cite[Table~16.1]{LPS2010}.
In this case, $G_v$ is a $\calC_3$-subgroup of~$G$.
However, since $\Gamma$ is $2$-arc-transitive, $G_v$ cannot be a $\calC_3$-subgroup of $G$ by~\cite[Lemma~5.1]{GLX2019}, a contradiction.

\smallskip
\noindent\textbf{Case~3:}
$L=\A_n$ and $L_v=\A_{n-1}$ for some $n\geq 5$, 
as in row~1 of~\cite[Table~16.2]{LPS2010}.
Then $|V(\Gamma)|=|L|/|L_v|=n$, and $G$ acts $2$-transitively on $V(\Gamma)$.
This implies that $\Gamma$ is a complete digraph, a contradiction.

\smallskip
\noindent\textbf{Case~4:}
$(L,G_v\cap L)=(\A_q,\Sy_{q-2})$, $(\A_p,\Z_p.\Z_{(p-1)/2})$, $(\A_{p+1},\PSL_2(p))$ or $(\A_{p^2+1},\PSL_2(p^2).\Z_2)$, as in rows~2--4 of~\cite[Table~16.2]{LPS2010}, where $q$ is a prime power and $p\geq 5$ is a prime.
Note that $G$ is an alternating or symmetric group whenever $L\ncong\A_6$.
Moreover, if $L=\A_6$, then $G_v\cap L=\PSL_2(5)$, and by~\cite[Page~4]{Atlas} we have $G\leq\Sy_6$ (since none of $\M_{10}$, $\PGL_2(9)$ or $\PGammaL_2(9)$ has a maximal subgroup whose intersection with $L$ is isomorphic to $\PSL_2(5)$).
Therefore, $G$ is an alternating or symmetric group.
According to the classification of maximal subgroups of the alternating and symmetric groups in~\cite[Theorem]{LPS1987}, we conclude that $G_v$ is of type~(a), (c), or~(f) of that theorem.
By~\cite[Lemmas~3.2 and~3.3]{PWY2020}, both types~(a) and~(c) yield $s=1$, while~\cite[Corollary~3.4]{GLX2019} shows that type~(f) cannot occur.
This completes the proof.
\end{proof}

Now we deal with the remaining cases in~\cite[Table~16.1--16.3]{LPS2010}.

\begin{lemma}\label{LemASFin}
Suppose that $(G, G_v, R)$ lies in~\cite[Tables~16.1--16.3]{LPS2010} except for rows~\emph{1--2} of~\cite[Table~16.1]{LPS2010} and rows~\emph{1--4} of~\cite[Table~16.2]{LPS2010}.
Then $s=1$.
\end{lemma}

\begin{proof}
Since $\Gamma$ is a $(G,2)$-arc-transitive digraph, $G_v$ has a factorization $G_v=G_{uv}G_{vw}$ as described in Lemma~\ref{LemPrimeFactn}.

\smallskip
\noindent\textbf{Case~1:}
$(L,L_v,R)=(\Omega_8^+(2),\Omega_7(2),\Sy_5)$, as in~\cite[Table~16.1]{LPS2010}.
Then $G=L.\calO$ with $L=\Omega_8^+(2)$, $L_v=\Omega_7(2)=\Sp_6(2)$, and $\calO \leq \Out(L)=\Sy_3$.
If $\calO=\Z_3$ or $\Sy_3$, inspection of the list of maximal subgroups of $G$ (see~\cite[Page~85]{Atlas}) shows that neither $\Omega_8^+(2).\Z_3$ nor $\Omega_8^+(2).S_3$ has a maximal subgroup whose intersection with $L$ is isomorphic to $\Sp_6(2)$.
Thus $\calO\leq\Z_2$, and hence $G_v=\Sp_6(2)\times\calO$ by~\cite[Page~85]{Atlas}.
Since $|L_v|=|\Sp_6(2)|=2^9 \cdot 3^4\cdot 5\cdot 7$, Lemma~\ref{LemPrimeFactn}\eqref{LemPrimeFactnB} implies that $2^5\cdot 3^2\cdot 5\cdot 7$ divides $|G_{uv}|=|G_{vw}|$.
Since $G_{uv}$ and $G_{vw}$ are not conjugate in $G_v$ by Lemma~\ref{LemPrimeFactn}\eqref{LemPrimeFactnC}, computations in \magma\cite{magma} show that there is no such factorization $G_v=G_{uv}G_{vw}$.

\smallskip
\noindent\textbf{Case~2:}
$(L,L_v,R)=(\Omega_8^+(4),\Omega_7(4),\PSL_2(16).\Z_4)$ with $G\geq L.\Z_2$, as in~\cite[Table~16.1]{LPS2010}.
Moreover, by~\cite[Table~8.50]{BHR2013}, we have $\calO\coloneqq G_v/L_v$ and hence $G_v=\Sp_6(4).\calO$ with $\calO\leq\Z_2^2$.
By Lemma~\ref{LemPrimeFactn}\eqref{LemPrimeFactnB}, $|G_{uv}|=|G_{vw}|$ is divisible by $2^9\cdot 3^2\cdot 5^2\cdot 7\cdot 13\cdot 17$.
Hence,~\cite[Corollary~5]{LPS2000} implies that both $G_{uv}$ and $G_{vw}$ contain $L_v=\Sp_6(4)$, whence $G_{uv}'=L_v=G_{vw}'$.
However, by Lemma~\ref{LemPrimeFactn}\eqref{LemPrimeFactnC}, there exists $g\in G\backslash G_v$ such that $G_{uv}^g=G_{vw}$.
It follows that $g\in\Nor_G(L_v)=G_v$, a contradiction.

\smallskip
\noindent\textbf{Case~3:}
$(L,L_v,R)=(\mathrm{He},\Sp_4(4).\Z_2, \Z_7^{1+2}.\Z_6)$, as in~\cite[Table~16.3]{LPS2010}.
Then $G=L.\calO$ with $\calO\leq\Z_2$.
Since $|L_v| = 2^9\cdot 3^2\cdot 5^2\cdot 17$, Lemma~\ref{LemPrimeFactn} gives that
\begin{equation}\label{eq1}
2^5\cdot 3\cdot 5\cdot 17\,\text{ divides }\,|G_{uv}|=|G_{vw}|.
\end{equation}
If $\calO=\Z_2$, then computation in \magma\cite{magma} shows that $G_v=\Aut(\Sp_4(4))$ has no non-conjugate subgroups $G_{uv}$ and $G_{vw}$ satisfying~\eqref{eq1}.
Thus $\calO=1$, and computation in \magma\cite{magma} shows that the only non-conjugate subgroups $G_{uv}$ and $G_{vw}$ in $G_v=\GammaSp_4(4)$ satisfying~\eqref{eq1} are not conjugate in $G$, contradicting Lemma~\ref{LemPrimeFactn}\eqref{LemPrimeFactnC}.

\smallskip
Similarly, we exclude the other cases, completing the proof.
\end{proof}

Combining Lemmas~\ref{LemASInf} and~\ref{LemASFin}, we obtain the following proposition.

\begin{proposition}\label{PropAS}
Let $\Gamma$ be a $G$-vertex-primitive $(G,2)$-arc-transitive Cayley digraph.
Then $G$ is not of type \AS.
\end{proposition}

Now we determine the upper bound on $s$ for $(G,s)$-arc-transitive digraphs such that $G$ acts primitively on $V(\Gamma)$ and is of type \SD or \CD.

\begin{proposition}\label{PropSDCD}
Let $\Gamma$ be a $G$-vertex-primitive Cayley digraph such that $G$ is of type \SD or \CD.
Then $\Gamma$ is $(G,2)$-arc-transitive if and only if $\Gamma$ is isomorphic to $\Gamma(T)^m$ for some nonabelian simple group $T$ and some positive integer $m$, where $\Gamma(T)$ is the digraph defined in Example~$\ref{ExampleSD}$.
Moreover, $\Gamma$ is not $3$-arc-transitive.
\end{proposition}

\begin{proof}
By~\cite[Theorem~1.2~and~Corollary~1.4]{GX2018}, a $(G,2)$-arc-transitive digraph $\Gamma$ with $G$ primitive of type \SD or \CD is isomorphic to the direct product of $m$ copies of $\Gamma(T)$ given in Example~\ref{ExampleSD},
where $m$ is a positive integer and $T$ is a nonabelian simple group.
By Lemma~\ref{SDisCayley}, $\Gamma(T)$ is a $2$-arc-transitive Cayley digraph that is not $3$-arc-transitive.
In particular, $\Aut(\Gamma(T))$ contains a regular subgroup, say $H$.
It follows that $\Gamma(T)^m$ is $2$-arc-transitive but not $3$-arc-transitive, and $H^m$ is a regular subgroup of $\Aut(\Gamma(T)^m)$.
Therefore, $\Gamma$ is a $2$-arc-transitive Cayley digraph that is not $3$-arc-transitive, which completes the proof.
\end{proof}

\section{Type \PA}\label{SecPA}

To prove Theorem~\ref{ThmMain} for the \PA case, it suffices to prove the following proposition.

\begin{proposition}\label{PropPA}
Let $\Gamma$ be a $G_0$-vertex-primitive $(G_0,2)$-arc-transitive Cayley digraph.
Then $G_0$ is not of type \PA.
\end{proposition}

Suppose for a contradiction that $\Gamma$ is a $G_0$-vertex-primitive $(G_0,2)$-arc-transitive Cayley digraph such that $G_0$ is of type \PA.
By~\cite[Theorem~1.3]{GX2018}, $\Gamma\cong\Sigma^m$ for some $(H,2)$-arc-transitive digraph $\Sigma$ and some integer $m\geq2$, where $H$ is primitive of type \AS such that $\Soc(H)^m\leq G_0\leq H\wr\Sy_m$.
Since $G_0$ has a regular subgroup, it follows from~\cite[Corollary~3(iv)]{LPS2000} that $G_0\leq G\wr\Sy_m$ in product action on $V(\Sigma)^m$, where $G$ is primitive on $V(\Sigma)$ of type \AS with socle $\Soc(H)$, and $G$ contains a transitive core-free subgroup.
Thus, we may proceed under the following hypothesis.

\begin{hypothesis}\label{Hyp}
Let $\Sigma$ be an $(H,2)$-arc-transitive digraph such that $H$ is vertex-primitive of type \AS, let $L=\Soc(H)$, and let $G$ be an almost simple primitive permutation group on $V(\Sigma)$ with socle $L$ such that $G$ has a transitive core-free subgroup $X$ and that $G\wr\Sy_m$ in product action on $V(\Sigma)^m$ has a regular subgroup.
Take $v\in V(\Sigma)$.
\end{hypothesis}

Under Hypothesis~\ref{Hyp}, it follows from~\cite[Corollary]{LPS1996} that $G_v=A$ or $B$ for some triple $(G,A,B)$ described in~\cite[Corollary~5 and~Tables~1--6]{LPS1990} or~\cite[Table~I]{LPS1996}.
We analyze these triples according to the type of the simple group $L$ in the subsequent subsections, which together constitute the proof of Proposition~\ref{PropPA}.
Recall the definition of $\calC_i$, $\mathcal{S}$, $P_i$, $N_i$, and $N_i^\pm$ given in~\cite[Pages~4--5]{LPS1990}.

\subsection{Non-classical groups}
\ \vspace{1mm}

In this subsection, we deal with non-classical groups $L$.
The cases where $L$ is an alternating group, an exceptional group of Lie type and a sporadic simple group are discussed in Lemmas~\ref{LemA},~\ref{LemExcep} and~\ref{LemSporadic} respectively.

\begin{lemma}\label{LemA}
Under Hypothesis~$\ref{Hyp}$, $L$ is not isomorphic to $\A_n$ with $n\geq 5$.
\end{lemma}

\begin{proof}
Suppose for a contradiction that $L=\A_n$ with $n\geq 5$.
Since $\Sigma$ is $(G,2)$-arc-transitive,~\cite[Theorem~1.3]{CCGLPX2025} asserts that one of the following holds:
\begin{enumerate}[\rm(i)]
\item $T\trianglelefteq G_v\leq\Aut(T)$ with $T=\Sp_4(2^f)$ or $\POmega_8^+(q)$, where $f\geq2$ and $q$ is a prime power;
\item $G_v=(T^k.(\Out(T)\times\Sy_k))\cap G$ with $n=|T|^{k-1}$ for some nonabelian simple group $T$ and integer $k\geq2$.
\end{enumerate}
Since $X$ is a core-free subgroup of $G$ that is transitive on $V(\Sigma)$, we obtain $G=XG_v$, and it follows from~\cite[Tables~7.3--7.4]{Cameron1999} and~\cite[Theorem~1]{Kantor1972} that $\A_{n-1}\leq X\leq\Sy_{n-1}$.
In particular, $X$ has a unique nonsolvable composition factor $\A_{n-1}$.
In either case~(i) or~(ii), it is straightforward to verify that $|\A_{n-1}|>|G|/|G_v|$.
For case~(i), since $G$ is $\A_n$ or $\Sy_n$ while $n$ is the index of some maximal subgroup of $G_v$, it follows from~\cite[Corollary~5]{LPS2000} that there exists $p\in\pi(G)\setminus\pi(G_v)$, and so $p$ is an odd prime such that $|G|_p=|\A_{n-1}|_p$.
For case~(ii), any prime $p$ with $n/2<p<n$ satisfies $p\in\pi(G)\setminus\pi(G_v)$ and $|G|_p=|\A_{n-1}|_p$.
Therefore, it follows from Lemma~\ref{LemRglr} that $G\wr\Sy_m$ has no regular subgroups for either case, contradicting Hypothesis~\ref{Hyp}.
\end{proof}

\begin{lemma}\label{LemExcep}
Under Hypothesis~$\ref{Hyp}$, $L$ is not isomorphic to an exceptional group of Lie type.
\end{lemma}

\begin{proof}
Suppose for a contradiction that $L$ is an exceptional group of Lie type.
Since $X$ is a core-free subgroup of $G$ that is transitive on $V(\Sigma)$, the factorization $G=XG_v$ is described in~\cite[Table~5]{LPS1990} (see also~\cite[Theorem~1]{HLS1987}).
In particular, $(G_v)^{(\infty)}$ is quasisimple.
By~\cite[Lemma~3.3]{YFX2023}, $L=\mathrm{G}_2(3)$, and $L_v=\SL_3(3)$ or $\SL_3(3).\Z_2$.
However, $|G|/|G_v|\leq|\mathrm{G}_2(3)|/|\SL_3(3)|<30758154560$, contradicting~\cite[Theorem~1.1]{YFX2023}.
\end{proof}

\begin{lemma}\label{LemSporadic}
Under Hypothesis~$\ref{Hyp}$, $L$ is not isomorphic to a sporadic simple group.
\end{lemma}

\begin{proof}
Suppose for a contradiction that $L$ is a sporadic simple group.
Since $X$ is a core-free subgroup of $G$ that is transitive on $V(\Sigma)$, the factorization $G=XG_v$ is classified in~\cite[Table~6]{LPS1990}.
However, calculation shows that $|G|/|G_v|<30758154560$ for each case, which is a contradiction by~\cite[Theorem~1.1]{YFX2023}.
\end{proof}

\subsection{Linear groups}
\ \vspace{1mm}

In this subsection, we consider the case where $L$ is a linear group.

\begin{lemma}\label{LemL}
Under Hypothesis~$\ref{Hyp}$, $L$ is not isomorphic to $\PSL_n(q)$.
\end{lemma}

\begin{proof}
Suppose for a contradiction that $L=\PSL_n(q)$.
Since~\cite[Theorem~1.1]{YFX2023} asserts that $|\Aut(\PSU_n(q))|\geq|V(\Sigma)|\geq30758154560$, we have $(n,q)\neq(3,3)$, $(3,4)$, $(3,8)$, $(5,2)$ or $(2,q_0)$ for any $q_0\leq59$.
Then as $G=XG_v$, we see from~\cite[Theorem~A]{LPS1990} and~\cite[Theorem]{LPS1996} that one of the following cases occurs.

\smallskip
\noindent\textbf{Case~1:}
$G_v\in\calC_1$.
Then either $L_v=P_k$ with $k=1$ or $n-1$, or $L_v=\mathrm{Stab}(V_1\oplus V_{n-1})$ with $n\geq4$ even.
If $L_v=P_k$, then $G_v$ is a maximal parabolic subgroup, which is excluded by~\cite[Lemma~2.13]{YFX2023}.
Hence $L_v=\mathrm{Stab}(V_1\oplus V_{n-1})$.
By~\cite[Proposition~4.1.4]{KL1990},
\[
L_v=\Z_{(q-1)/\gcd(n,q-1)}.\PSL_{n-1}(q).\Z_{\gcd(n-1,q-1)}.
\]
In particular, $(G_v)^{(\infty)}$ is quasisimple.
It follows from~\cite[Lemma~3.3]{YFX2023} that this case does not occur, a contradiction.

\smallskip
\noindent\textbf{Case~2:}
$G_v\in\calC_3$.
This case is impossible by~\cite[Lemma~5.1]{GLX2019}.

\smallskip
\noindent\textbf{Case~3:}
$G_v\in\calC_8$.
Then $(L_v)^{(\infty)}=\PSp_n(q)$ with $n\geq4$ even.
It follows from~\cite[Lemma~3.3]{YFX2023} that $q$ is even, $n=4$, and $(G_{uv})^{(\infty)}\cong(G_{vw})^{(\infty)}=\PSL_2(q^2)$.
Moreover, since $G_v=G_{uv}G_{vw}$ and $(G_v)^{(\infty)}=\PSp_4(q)$, we derive from~\cite[Theorem~A]{LPS1990} that, interchanging $G_{uv}$ and $G_{vw}$ if necessary, $(G_{uv})^{(\infty)}=\PSL_2(q^2)$ is a $\calC_3$-subgroup of $(G_v)^{(\infty)}$ and $(G_{vw})^{(\infty)}=\Omega_4^-(q)$ is a $\calC_8$-subgroup of $(G_v)^{(\infty)}$.
Let $\varphi$ be the natural projection from $\GL_4(q)$ to $\PGL_4(q)=L$ (note that $\gcd(4,q-1)=1$), and let $S=(G_{uv})^{(\infty)}$ and $T=(G_{vw})^{(\infty)}$.
Then $S=(\SL_2(q^2))^\varphi$, and $T=\Omega_4^-(q)$ acts on $\mathbb{F}_q^4$ by stabilizing a non-degenerate quadratic form of minus type, whose associated bilinear form is the symplectic form preserved by $(L_v)^{(\infty)}=\Sp_4(q)$.
It follows that 
\[
\Cen_L(S)
=(\Cen_{\GL_4(q)}(\SL_2(q^2)))^\varphi
\geq(\Cen_{\GL_2(q^2)}(\SL_2(q^2)))^\varphi
\geq(\ZZ(\GL_2(q^2)))^\varphi
=\Z_{q+1}.
\]
However, $T$ acts absolutely irreducibly on $\mathbb{F}_q^4$ by~\cite[Proposition~2.10.6]{KL1990}, and hence $\Cen_L(T)=1$ by~\cite[Proposition~4.0.5(ii)]{KL1990}.
This leads to
\[
|\Cen_{G^{(\infty)}}((G_{uv})^{(\infty)})|
=|\Cen_L(S)|
\neq|\Cen_L(T)|
=|\Cen_{G^{(\infty)}}((G_{vw})^{(\infty)})|,
\]
which implies that $G_{uv}$ and $G_{vw}$ are not conjugate in $G$, contradicting Lemma~\ref{LemPrimeFactn}\eqref{LemPrimeFactnC}.
\end{proof}

\subsection{Unitary groups}
\ \vspace{1mm}

In this subsection, we handle the case where $L$ is a unitary group.

\begin{lemma}\label{LemU}
Under Hypothesis~$\ref{Hyp}$, $L$ is not isomorphic to $\PSU_n(q)$ with $n\geq 3$.
\end{lemma}

\begin{proof}
Suppose for a contradiction that $L=\PSU_n(q)$ with $n\geq 3$.
Since~\cite[Theorem~1.1]{YFX2023} asserts that $|\Aut(\PSU_n(q))|\geq|V(\Sigma)|\geq30758154560$, we have $(n,q)\neq(3,3)$, $(3,5)$, $(3,8)$, $(4,2)$, $(4,3)$ or $(6,2)$.
Then as $G=XG_v$, we see from~\cite[Theorem~A]{LPS1990} and~\cite[Theorem]{LPS1996} that one of the following cases occurs.

\smallskip
\noindent\textbf{Case~1:}
$G_v\in\calC_1$.
Then $L_v=P_k$ (where $1\leq k\leq n/2$) or $N_1$ (with $n=2m$ even).
If $L_v=P_k$, then $G_v$ is a maximal parabolic subgroup, which is excluded by~\cite[Lemma~2.13]{YFX2023}.
Thus, $L_v=N_1$ with $n=2m$.
By~\cite[Proposition~4.1.4]{KL1990},
\[
L_v=\Z_{(q+1)/\gcd(2m,q+1)}.\PSU_{2m-1}(q).\Z_{\gcd(2m-1,q+1)}.
\]
In particular, either $(m,q)=(2,2)$ or $(G_v)^{(\infty)}$ is quasisimple.
If $(m,q)=(2,2)$, then $L=\PSU_4(2)$, $L_v=\SU_3(2).\Z_2$ and $|V(\Sigma)|=60$,
which is impossible by~\cite[Theorem~1.1]{YFX2023}.
Hence $(G_v)^{(\infty)}$ is quasisimple.
Then~\cite[Lemma~3.3]{YFX2023} implies that $(m,q)=(2,8)$, and so $|G|/|G_v|=|\PSU_4(8)|/|\Z_9.\PSU_3(8).\Z_3|<30758154560$, contradicting~\cite[Theorem~1.1]{YFX2023}.

\smallskip
\noindent\textbf{Case~2:}
$G_v\in\calC_2$.
In this case, $n=2m$, and according to~\cite[Proposition 4.2.4]{KL1990}, either $L_v=\PSL_m(4).\Z_2$ with $q=2$, or $L_v=\Z_3.\PSL_m(16).\Z_a.\Z_2$ with $q=4$ and $a=\gcd(m,255)/\gcd(m,17)$.
In particular, $(G_v)^{(\infty)}$ is quasisimple, and so~\cite[Lemma~3.3]{YFX2023} gives $L_v=\PSL_3(4).\Z_2$ with $(m,q)=(3,4)$, and interchanging $G_{uv}$ and $G_{vw}$ if necessary, $\pi(G_{uv})\subseteq\{2,3,7\}$ and $\pi(G_{vw})=\{2,3,5\}$.
It follows that $5\in\pi(G_v)$ but $5\notin\pi(G_{uv})$, contradicting Lemma~\ref{LemPrimeFactn}\eqref{LemPrimeFactnB}.

\smallskip
\noindent\textbf{Case~3:}
$G_v\in\calC_5$.
In this case, $n=2m$, and~\cite[Proposition 4.5.6]{KL1990} gives $L_v=\PSp_{2m}(q).\Z_a$ with $a=\gcd(2,q-1)\gcd(m,q+1)/\gcd(2m,q+1)$.
Then we derive from~\cite[Lemma~3.3]{YFX2023} that $L_v=\PSp_4(q)$ with $(m,q)=(2,2^f)$, and interchanging $G_{uv}$ and $G_{vw}$ if necessary, $(G_{uv}\cap\Soc(G_v),G_{vw}\cap\Soc(G_v))=(\PSL_2(q^2).\Z_2,\PSL_2(q^2))$ or $(\PSL_2(q^2).\Z_2,\PSL_2(q^2).\Z_2)$.
By~\cite[Theorem~1.1]{LX2022} and~\cite[Theorem~4.1]{LWX2024}, we derive from $G=XG_v$ that $X^{(\infty)}=\SU_3(q)$.
Hence $X$ has a unique nonsolvable composition factor $T=\PSU_3(q)$.
It is straightforward to verify that $|T|>|G|/|G_v|$, and that there exists a prime $p\in\ppd(q,6)$ satisfying $p\in\pi(G)\setminus\pi(G_v)$ and $|T|_p=|G|_p$.
It then follows from Lemma~\ref{LemRglr} that $G\wr\Sy_m$ has no regular subgroups, a contradiction.

\smallskip
\noindent\textbf{Case~4:}
$G_v\in\mathcal{S}$.
Then $(L,L_v)=(\PSU_9(2),\mathrm{J}_3)$ or $(\PSU_{12}(2),\mathrm{Suz})$ as in~\cite[Table~3]{LPS1990}.
Since $(G_v)^{(\infty)}=\mathrm{J}_3$ or $\mathrm{Suz}$, we obtain a contradiction by~\cite[Lemma~3.3]{YFX2023}.
\end{proof}

\subsection{Symplectic groups}
\ \vspace{1mm}

The analysis on symplectic groups $L$ is analogous to that on unitary groups.

\begin{lemma}\label{LemSp}
$L$ is not isomorphic to $\PSp_{2m}(q)$ with $m\geq2$.
\end{lemma}

\begin{proof}
Suppose for a contradiction that $L=\PSp_{2m}(q)$ with $m\geq2$.
By~\cite[Theorem~1.1]{YFX2023}, $(m,q)\neq(2,3)$, $(2,4)$ or $(3,3)$.
Then as $G=XG_v$, we see from~\cite[Theorem~A]{LPS1990} and~\cite[Theorem]{LPS1996} that one of the following cases occurs.

\smallskip
\noindent\textbf{Case~1:}
$G_v\in\calC_1$.
Then either $L_v=P_i$ with $i\in\{1,m\}$, or $L_v=N_2$ with $q\in\{2,4\}$.
If $L_v=P_1$ or $P_m$, then $G_v$ is a maximal parabolic subgroup, which is excluded by~\cite[Lemma~2.13]{YFX2023}.
Thus, $L_v=N_2$ with $q\in\{2,4\}$.
By~\cite[Proposition~4.1.3]{KL1990}, we have $L_v=\Sp_2(q)\times\Sp_{2m-2}(q)$.
If $q=2$, then $G_v^{(\infty)}$ is quasisimple, it follows from~\cite[Lemma~3.3]{YFX2023} that this case does not occur.
Hence $q=4$, and since~\cite[Theorem~1.1]{YFX2023} asserts that $|V(\Sigma)|=|G|/|G_v|\geq30758154560$, we conclude that $m\geq6$.
Then $L$ has index at most $2$ in $G$, and so does $L_v=\Sp_2(4)\times\Sp_{2m-2}(4)$ in $G_v$.
By~\cite[Theorem~1.1]{LX2019}, it follows from $G=XG_v$ that either $(L,X^{(\infty)})=(\Sp_{12}(4),\mathrm{G}_2(16))$ with $m=6$, or $(L,X^{(\infty)})=(\Sp_{2m}(4),\Sp_m(16))$ with $m$ even.
In particular, $|X^{(\infty)}|>|G|/|G_v|$.
Moreover, for $(L,X^{(\infty)})=(\Sp_{12}(4),\mathrm{G}_2(16))$, the prime $241\in\pi(G)\setminus\pi(G_v)$ is such that $|\mathrm{G}_2(16)|_{241}=|\Sp_{12}(4)|_{241}=|G|_{241}$;
for $(L,X^{(\infty)})=(\Sp_{2m}(4),\Sp_m(16))$, any prime $p\in\ppd(4,2m)$ satisfies $p\in\pi(G)\setminus\pi(G_v)$ and $|\Sp_m(16)|_p=|\Sp_{2m}(4)|_p=|G|_p$.
Thus, Lemma~\ref{LemRglr} asserts that $G\wr\Sy_m$ has no regular subgroups, a contradiction.

\smallskip
\noindent\textbf{Case~2:}
$G_v\in\calC_2$.
Then $L_v=\Sp_m(q)\wr\Sy_2$ with $m$ and $q$ even, and it follows from~\cite[Theorem~1.1]{LX2019} that $X\cap L=\POmega_{2m}^-(q)$.
Hence $X$ has a unique nonsolvable composition factor $T=\POmega_{2m}^-(q)$.
It is straightforward to verify that $|T|>|G|/|G_v|$, and that there exists a prime $p\in\ppd(q,2m)$ satisfying $p\in\pi(G)\setminus\pi(G_v)$ and $|T|_p=|G|_p$.
Thus, it follows from Lemma~\ref{LemRglr} that $G\wr\Sy_m$ has no regular subgroups, a contradiction.

\smallskip
\noindent\textbf{Case~3:}
$G_v\in\calC_3$.
Then $L_v=\Sp_{2a}(q^r)$, where $r$ is a prime and $m=ar$.
In particular, $G_v^{(\infty)}$ is quasisimple, it follows from~\cite[Lemma~3.3]{YFX2023} that $L_v=\Sp_4(q^r)$.
Since $G=XG_v$, it follows from~\cite[Theorem~1.1]{LX2022} and~\cite[Theorem~8.1]{LWX2024} that $X$ has a nonsolvable composition factor $T=\Sp_{2m-2}(q)$ or $\Omega_{2m}^\pm(q)$.
It is straightforward to verify that $|T|>|G|/|G_v|$, and that there exists a prime $p\in\ppd(q,2m-2)$ satisfying $p\in\pi(G)\setminus\pi(G_v)$ and $|T|_p=|G|_p$.
It then follows from Lemma~\ref{LemRglr} that $G\wr\Sy_m$ has no regular subgroups, a contradiction.

\smallskip
\noindent\textbf{Case~4:}
$G_v\in\calC_5\cup\calC_8\cup\mathcal{S}$, or $L=\Sp_8(2)$ as in~\cite[Table~3]{LPS1990}.
If $G_v\in\calC_5$, then since $G_v^{(\infty)}$ is quasisimple, it follows from~\cite[Lemma~3.3]{YFX2023} that $(L,L_v)=(\Sp_4(16),\Sp_4(4))$, and so $|V(\Sigma)|\leq|\Sp_4(16)|/|\Sp_4(4)|<30758154560$, contradicting~\cite[Theorem~1.1]{YFX2023}.
If $G_v\in\calC_8$, then~\cite[Proposition~4.23]{CGP2025} yields a contradiction.
If $L=\Sp_8(2)$ as in~\cite[Table~3]{LPS1990}, then $|V(\Sigma)|\leq|\Sp_8(2)|/|\PSL_2(17)|<30758154560$, contradicting~\cite[Theorem~1.1]{YFX2023}.
Thus, $G_v\in\mathcal{S}$ as in~\cite[Table~2]{LPS1990}, and one of the following holds:
\begin{itemize}
\item $m=2$, $q=2^{2c+1}\geq8$, and $L_v=\mathrm{Sz}(q)$;
\item $m=3$, $q=2^f$, and $L_v=\mathrm{G}_2(q)$.
\end{itemize}
Then $(G_v)^{(\infty)}$ is simple, and~\cite[Lemma~3.3]{YFX2023} gives a contradiction.
\end{proof}

\subsection{Orthogonal groups}
\ \vspace{1mm}

This subsection is devoted to the case where $L$ is an orthogonal group.
We divide these groups into four types: odd dimension, even dimension of minus type, even dimension at least $10$ of plus type, and $8$-dimensional plus type.
These types are analyzed sequentially in Lemmas~\ref{LemO}--\ref{LemO8+} respectively.

\begin{lemma}\label{LemO}
Under Hypothesis~$\ref{Hyp}$, $L$ is not isomorphic to $\Omega_{2m+1}(q)$ with $m\geq 3$ and $q$ odd.
\end{lemma}

\begin{proof}
Suppose for a contradiction that $L=\Omega_{2m+1}(q)$ with $m\geq 3$ and $q$ odd.
It follows from~\cite[Theorem~1.1]{YFX2023} that $(m,q)\neq(3,3)$.
Then since $G=XG_v$, we see from~\cite[Theorem~A]{LPS1990} and~\cite[Theorem]{LPS1996} that $G_v\in\calC_1\cup\mathcal{S}$.

Suppose that $G_v\in\calC_1$.
Then $L_v=P_1$, $P_m$, $N_1^\epsilon$ or $N_2^\epsilon$ as given in~\cite[Tables~1 and~2]{LPS1990}, where $\epsilon\in\{+,-\}$.
In particular, if $L_v=N_1^+$, then $m=3$.
If $L_v=P_1$ or $P_m$, then $G_v$ is a maximal parabolic subgroup, which is excluded by~\cite[Lemma~2.13]{YFX2023}.
Therefore, $L_v=N_1^\epsilon$ or $N_2^\epsilon$, where $N_1^\epsilon=\Omega_{2m}^\epsilon(q).2$ and $N_2^\epsilon=(\Omega_2^\epsilon(q)\times\Omega_{2m-1}(q)).[4]$ as given in~\cite[Proposition~4.1.6]{KL1990}.
However, it follows that $(G_v)^{(\infty)}$ is quasisimple, which is not possible by~\cite[Lemma~3.3]{YFX2023}.

Thus, $L_v\in\mathcal{S}$ and $G_v$ is almost simple with socle $\mathrm{G}_2(q)$, $\PSp_6(q)$ or $\mathrm{F}_4(q)$ as in~\cite[Table~2]{LPS1990}.
This contradicts~\cite[Lemma~3.3]{YFX2023}, completing the proof.
\end{proof}

\begin{lemma}\label{LemO-}
Under Hypothesis~$\ref{Hyp}$, $L$ is not isomorphic to $\POmega_{2m}^-(q)$ with $m\geq 4$.
\end{lemma}

\begin{proof}
Suppose for a contradiction that $L=\POmega_{2m}^-(q)$ for some $m\geq 4$.
Then since $G=XG_v$, we derive from~\cite[Theorem~A]{LPS1990} and~\cite[Theorem]{LPS1996} that one of the following cases occurs.

\smallskip
\noindent\textbf{Case~1:}
$G_v\in\calC_1$.
In this case, $L_v=P_1$, $N_1$ or $N_2^+$.
If $L_v=P_1$, then $G_v$ is a maximal parabolic subgroup, which is excluded by~\cite[Lemma~2.13]{YFX2023}.
If $L_v=N_1$, then by~\cite[Proposition~4.1.6]{KL1990}, $L_v=\Omega_{2m-1}(q)$, and~\cite[Lemma~3.3]{YFX2023} asserts that this is not possible.
Therefore, $L_v=N_2^+$, where $m$ is odd and $q=4$.
Then~\cite[Proposition~4.1.6]{KL1990} implies that $(G_v)^{(\infty)}=\Omega_{2m-2}^-(4)$ is quasisimple, which is impossible by~\cite[Lemma~3.3]{YFX2023}.

\smallskip
\noindent\textbf{Case~2:}
$G_v\in\calC_3$.
Then it follows from~\cite[Propositions~4.3.16 and~4.3.18]{KL1990} that $(G_v)^{(\infty)}=\POmega_m^-(q^2)$ or $\Z_{(q+1)/\gcd(q+1,4)}.\PSU_m(q)$.
It follows that $(G_v)^{(\infty)}$ is quasisimple, and~\cite[Lemma~3.3]{YFX2023} shows that this is not possible.

\smallskip
\noindent\textbf{Case~3:}
$G_v\in\mathcal{S}$.
Then $L=\Omega_{10}^-(2)$ and $L_v=\A_{12}$ as in~\cite[Table~3]{LPS1990}.
As a consequence, $|G|/|G_v|=|\Omega_{10}^-(2)|/|\A_{12}|<30758154560$, contradicting~\cite[Theorem~1.1]{YFX2023}.
\end{proof}

\begin{lemma}\label{LemO+}
Under Hypothesis~$\ref{Hyp}$, $L$ is not isomorphic to $\POmega_{2m}^+(q)$ with $m\geq 5$.
\end{lemma}

\begin{proof}
Suppose for a contradiction that $L=\POmega_{2m}^+(q)$ with $m\geq 5$.
Since $G=XG_v$, we derive from~\cite[Theorem~A]{LPS1990} and~\cite[Theorem]{LPS1996} that one of the following cases occurs.

\smallskip
\noindent\textbf{Case~1:}
$G_v\in\calC_1$.
Then $L_v\cong P_1$, $N_1$ or $N_2^\pm$ as in~\cite[Theorem~A]{LPS1990}.
If $L_v=P_1$, then $G_v$ is a maximal parabolic subgroup, which is excluded by~\cite[Lemma~2.13]{YFX2023}.
Thus, $L_v=N_1$ or $N_2^\pm$, and by~\cite[Proposition~4.1.6]{KL1990}, $(G_v)^{(\infty)}$ is quasisimple.
Then~\cite[Lemma~3.3]{YFX2023} implies that $L_v=N_2^+$ and $m=5$.
However, according to~\cite[Table~1]{LPS1990}, $q=4$.
This leads to $|G|/|G_v|\leq|\POmega_{10}^+(4)|/|\POmega_8^+(4)|<30758154560$, contradicting~\cite[Theorem~1.1]{YFX2023}.

\smallskip
\noindent\textbf{Case~2:}
$G_v\in\calC_2$.
In this case, $(G_v)^{(\infty)}$ is quasisimple with the unique nonsolvable composition factor $\PSL_m(q)$, contradicting~\cite[Theorem~3.3]{YFX2023}.

\smallskip
\noindent\textbf{Case~3:}
$G_v\in\calC_3$.
Then $q\in\{2,4\}$, and $(G_v)^{(\infty)}$ is quasisimple with the unique nonsolvable composition factor $\PSL_m(q)$ or $\Omega_m^+(q^2)$.
By~\cite[Lemma~3.3]{YFX2023}, $(G_v)^{(\infty)}=\Omega_8^+(q^2)$ with $m=8$.
By~\cite[Theorem~1.1]{LX2022} and~\cite[Theorem~4.1]{LWX2024}, we derive from $G=XG_v$ that $X^{(\infty)}=\Sp_{2m-2}(q)$.
It is straightforward to verify that $|\Sp_{2m-2}(q)|>|G|/|G_v|$, and that any $p\in\ppd(q,2m-2)$ satisfies $p\in\pi(G)\setminus\pi(G_v)$ and $|\Sp_{2m-2}(q)|_p=|G|_p$.
It then follows from Lemma~\ref{LemRglr} that $G\wr\Sy_m$ has no regular subgroups, a contradiction.

\smallskip
\noindent\textbf{Case~4:}
$G_v\in\calC_4$.
Then $L_v=\PSp_2(q)\otimes\PSp_m(q)$ with $m$ even and $q>2$, and we deduce from~\cite[Proposition 4.4.12]{KL1990} that $L_v$ has a subgroup $\PSp_2(q)\times\PSp_m(q)$ of index at most $2$.
If $q=3$, then $(G_v)^{(\infty)}=\PSp_m(3)$ is simple, which is not possible by~\cite[Lemma~3.3]{YFX2023}.
Therefore, $q\geq4$ and $(G_v)^{(\infty)}=\PSp_2(q)\times\PSp_m(q)$.
It follows from~\cite[Theorem~1.1]{LX2019} that $X\cap L=\Omega_{2m-1}(q)$.
Hence $X$ has a unique nonsolvable composition factor $T=\Omega_{2m-1}(q)$.
Since $|T|>|G|/|G_v|$ and any $p\in\ppd(q,2m-2)$ satisfies $p\in\pi(G)\setminus\pi(G_v)$ and $|T|_p=|G|_p$, it follows from Lemma~\ref{LemRglr} that $G\wr\Sy_m$ has no regular subgroups, a contradiction.

\smallskip
\noindent\textbf{Case~5:}
$G_v\in\mathcal{S}$.
Then $(L,(G_v)^{(\infty)})=(\Omega_{24}^+(2),\Co_1)$ or $(\POmega_{16}^+(q),\Omega_9(q))$.
In particular, $(G_v)^{(\infty)}$ is quasisimple, and~\cite[Lemma~3.3]{YFX2023} shows that this case is not possible.
\end{proof}

\begin{lemma}\label{LemO8+}
Under Hypothesis~$\ref{Hyp}$, $L$ is not isomorphic to $\POmega_8^+(q)$.
\end{lemma}

\begin{proof}
Suppose for a contradiction that $L=\POmega_8^+(q)$.
By~\cite[Theorem~1.1]{YFX2023}, $q\neq2$.
Then as $G=XG_v$, we see from~\cite[Theorem~A]{LPS1990} (see~\cite[\S2]{GGS2024} for the missing factorizations of $\Omega_8^+(4)$ and $\Omega_8^+(16)$) and~\cite[Theorem]{LPS1996} that one of the following cases occurs.

\smallskip
\noindent\textbf{Case~1:}
$G_v$ is a maximal parabolic subgroup.
This case is excluded by~\cite[Lemma~2.13]{YFX2023}.

\smallskip
\noindent\textbf{Case~2:}
$(G_v)^{(\infty)}$ is quasisimple with a nonsolvable composition factor isomorphic to $\POmega_6^\pm(q)$, $\Omega_7(q)$ or $\POmega_8^-(q_0)$ for some $q_0<q$.
It follows from~\cite[Lemma~3.3]{YFX2023} that this case is not possible.

\smallskip
\noindent\textbf{Case~3:}
$(G_v)^{(\infty)}=\PSp_2(q)\times\PSp_4(q)$ with $q$ odd.
We derive from~\cite[Theorem~1.1]{LX2019} that $X\cap L=\Omega_7(q)$, and so $X$ has a unique nonsolvable composition factor $T=\Omega_7(q)$.
Since $|T|>|G|/|G_v|$ and any $p\in\ppd(q,6)$ satisfies $p\in\pi(G)\setminus\pi(G_v)$ and $|T|_p=|G|_p$, it follows from Lemma~\ref{LemRglr} that $G\wr\Sy_m$ has no regular subgroups, a contradiction.

\smallskip
\noindent\textbf{Case~4:}
$L=\POmega_8^+(3)$ and $L_v=\Z_2^6.\A_8$.
Then $|G|/|G_v|=|\POmega_8^+(3)|/|\Z_2^6.\A_8|<30758154560$, contradicting~\cite[Theorem~1.1]{YFX2023}.
\end{proof}

\end{document}